\documentclass[10pt]{article}
\usepackage{amsmath,amsfonts,latexsym, amssymb}
\usepackage{mathrsfs}
\usepackage{pstricks}
\usepackage{pst-plot}
\usepackage{pst-node}

\def\endproofbox{\hskip 1.3em\hfill\rule{6pt}{6pt}}
\newenvironment{proof}%
{%
\noindent{\it Proof.}
}%
{%
 \quad\hfill\endproofbox\vspace*{2ex}
}
 \newcommand{\rmv}[1]{}

 \def\f{\varphi}
 
 \def\la{\lambda}

 \def\sC{\mathscr{C}}

\newtheorem{thm}{Theorem}[section]
\newtheorem{theorem}[thm]{Theorem}

\newtheorem{lemma}[thm]{Lemma}

\def\endproofbox{\hskip 1.3em\hfill\rule{6pt}{6pt}}

   \rmv{

\newenvironment{proof}%
{%
\noindent{\it Proof.}
}%
{%
 \quad\hfill\endproofbox\vspace*{2ex}
}

    }

\def\to{\rightarrow}

\def\bl{$\bullet$}

 \def\Pow{\textsf{Pow}}

 \rmv{

 \headsep0.3in     

 \pagestyle{myheadings}

 \markboth{Robert E. Jamison \hspace{.4 in}
 >>> Latex ignores this <<<  \hspace{.4 in}
 >>> Latex ignores this <<<  Version:  $8^{th}$ MAY, 2008}
  {Beyerl and Jamison  \hspace{.1 in} Balanced Interval Graphs \hspace{.1 in}}

 }

 \author{Jeffery J. Beyerl
 \thanks{Department of Mathematics, Clemson University, Clemson, SC 29634-0975
 \mbox{ email: \textsf{jbeyerl@clemson.edu}}}
 \and
 Robert E. Jamison
 \thanks{Department of Mathematics, Clemson University, Clemson, SC 29634-0975
 \mbox{ email: \textsf{rejam@clemson.edu}}}
 \thanks{Affiliated Professor, University of Haifa}
 \and
 J. Bowman Light
  \thanks{Department of Mathematics, Clemson University, Clemson, SC 29634-0975
  \mbox{ email: \textsf{jlight@clemson.edu}}}
 }

 \title{Depth in Bingo Closure}

\begin{document}

     \date{}

 \maketitle


 \begin{abstract}

 Bingo is played on a $5\times 5$ grid. Take the 25 squares to be the ground set of a closure system in
 which square $s$ is {\em dependent} on a set $S$ of squares iff $s$ completes
 a {\em line} --- a row, column, or diagonal --- with squares that are already in $S$.
 The closure of a set $S$ is obtained via an iterative process in
 which, at each stage, the squares dependent upon the current state are added.
 In this paper we establish for the $n \times n$ Bingo board the
 maximum number of steps required in this closure process.
 \end{abstract}


    \section{Introduction}  \label{intro}

 If you are playing a game of Bingo and you find that you already have 4 squares on any row, column, or
 diagonal, then the fifth square on that line takes of special interest.
 If you get that fifth square, it {\em completes} a line and you
 win.  Convexity, and more generally closure, is based on the idea
 of filling gaps, closing holes, completing some set.  A common and
 convenient way to express this is in terms on of {\em dependency}.
 Here linear algebra is a model:  a subspace is a set of vectors
 which contains all vectors dependent on it.


 \begin{figure}
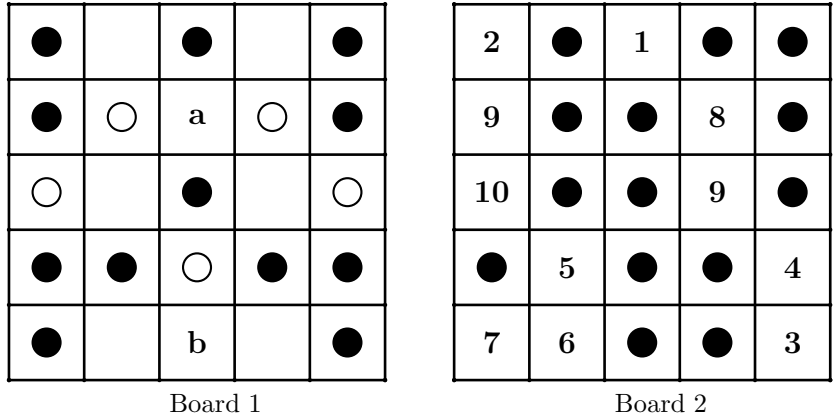

 \begin{center}
 \begin{tabular}{cc}


    \pspicture(5.5,5.5)
    \cnode[](0,0){0}{x00}       \rput(2.5,3.5){\large {\bf a}}     \rput(2.5,.5){\large {\bf b}}
    \cnode[](0,1){0}{x01}
    \cnode[](0,2){0}{x02}
    \cnode[](0,3){0}{x03}
    \cnode[](0,4){0}{x04}
    \cnode[](0,5){0}{x05}

    \cnode[fillstyle=solid,fillcolor=black](.5,4.5){.2}{a1}
    \cnode[fillstyle=solid,fillcolor=black](.5,1.5){.2}{a4}
    \cnode[fillstyle=solid,fillcolor=black](.5,.5){.2}{a5}
    \cnode[fillstyle=solid,fillcolor=black](1.5,1.5){.2}{b4}
    \cnode[fillstyle=solid,fillcolor=black](2.5,2.5){.2}{c3}
    \cnode[fillstyle=solid,fillcolor=black](3.5,1.5){.2}{d4}
    \cnode[fillstyle=solid,fillcolor=black](4.5,1.5){.2}{e4}
    \cnode[fillstyle=solid,fillcolor=black](4.5,.5){.2}{e5}
    \cnode[fillstyle=solid,fillcolor=black](4.5,4.5){.2}{a5}
    \cnode[fillstyle=solid,fillcolor=black](4.5,3.5){.2}{e4}

    \cnode[](1.5,3.5){.2}{b2}
    \cnode[](2.5,1.5){.2}{c4}
    \cnode[](3.5,3.5){.2}{d2}
    \cnode[](4.5,2.5){.2}{e3}

    \cnode[fillstyle=solid,fillcolor=black](.5,3.5){.2}{a2}
    \cnode[fillstyle=solid,fillcolor=black](2.5,4.5){.2}{c1}

                \cnode[](.5,2.5){.2}{a3}

    \cnode[](5,0){0}{y00}
    \cnode[](5,1){0}{y01}
    \cnode[](5,2){0}{y02}
    \cnode[](5,3){0}{y03}
    \cnode[](5,4){0}{y04}
    \cnode[](5,5){0}{y05}

    \cnode[](0,5){0}{t00}
    \cnode[](1,5){0}{t01}
    \cnode[](2,5){0}{t02}
    \cnode[](3,5){0}{t03}
    \cnode[](4,5){0}{t04}
    \cnode[](5,5){0}{t05}

    \cnode[](0,0){0}{s00}
    \cnode[](1,0){0}{s01}
    \cnode[](2,0){0}{s02}
    \cnode[](3,0){0}{s03}
    \cnode[](4,0){0}{s04}
    \cnode[](5,0){0}{s05}

    \ncline[linewidth=1pt]{x00}{y00}
    \ncline[linewidth=1pt]{x01}{y01}
    \ncline[linewidth=1pt]{x02}{y02}
    \ncline[linewidth=1pt]{x03}{y03}
    \ncline[linewidth=1pt]{x04}{y04}
    \ncline[linewidth=1pt]{x05}{y05}

    \ncline[linewidth=1pt]{s00}{t00}
    \ncline[linewidth=1pt]{s01}{t01}
    \ncline[linewidth=1pt]{s02}{t02}
    \ncline[linewidth=1pt]{s03}{t03}
    \ncline[linewidth=1pt]{s04}{t04}
    \ncline[linewidth=1pt]{s05}{t05}

 \endpspicture

 &


    \pspicture(5.5,5.5)
    \cnode[](0,0){0}{x00}
    \cnode[](0,1){0}{x01}
    \cnode[](0,2){0}{x02}
    \cnode[](0,3){0}{x03}
    \cnode[](0,4){0}{x04}
    \cnode[](0,5){0}{x05}

            \rput(2.5,4.5){\large {\bf 1}}
    \cnode[fillstyle=solid,fillcolor=black](1.5,4.5){.2}{b1}
    \cnode[fillstyle=solid,fillcolor=black](3.5,4.5){.2}{d1}
            \rput(.5,4.5){\large {\bf 2}}
    \cnode[fillstyle=solid,fillcolor=black](4.5,4.5){.2}{e1}

    \cnode[fillstyle=solid,fillcolor=black](1.5,3.5){.2}{a2}
    \cnode[fillstyle=solid,fillcolor=black](2.5,3.5){.2}{a2}
            \rput(4.5,.5){\large {\bf 3}}
            \rput(4.5,1.5){\large {\bf 4}}
    \cnode[fillstyle=solid,fillcolor=black](4.5,3.5){.2}{b2}
             \rput(1.5,1.5){\large {\bf 5}}

    \cnode[fillstyle=solid,fillcolor=black](1.5,2.5){.2}{b2}
        \rput(1.5,.5){\large {\bf 6}}
        \rput(.5,.5){\large {\bf 7}}
    \cnode[fillstyle=solid,fillcolor=black](2.5,2.5){.2}{c3}
    \cnode[fillstyle=solid,fillcolor=black](4.5,2.5){.2}{a2}
            \rput(3.5, 3.5){\large {\bf 8}}

        \rput(.5, 3.5){\large {\bf 9}}
        \rput(3.5,2.5){\large {\bf 9}}
    \cnode[fillstyle=solid,fillcolor=black](.5,1.5){.2}{d4}
    \cnode[fillstyle=solid,fillcolor=black](2.5,1.5){.2}{a2}
    \cnode[fillstyle=solid,fillcolor=black](3.5,1.5){.2}{b4}

    \cnode[fillstyle=solid,fillcolor=black](2.5,.5){.2}{d4}
    \cnode[fillstyle=solid,fillcolor=black](3.5,.5){.2}{b4}
        \rput(.5, 2.5){\large {\bf 10}}

    \cnode[](5,0){0}{y00}
    \cnode[](5,1){0}{y01}
    \cnode[](5,2){0}{y02}
    \cnode[](5,3){0}{y03}
    \cnode[](5,4){0}{y04}
    \cnode[](5,5){0}{y05}

    \cnode[](0,5){0}{t00}
    \cnode[](1,5){0}{t01}
    \cnode[](2,5){0}{t02}
    \cnode[](3,5){0}{t03}
    \cnode[](4,5){0}{t04}
    \cnode[](5,5){0}{t05}

    \cnode[](0,0){0}{s00}
    \cnode[](1,0){0}{s01}
    \cnode[](2,0){0}{s02}
    \cnode[](3,0){0}{s03}
    \cnode[](4,0){0}{s04}
    \cnode[](5,0){0}{s05}

    \ncline[linewidth=1pt]{x00}{y00}
    \ncline[linewidth=1pt]{x01}{y01}
    \ncline[linewidth=1pt]{x02}{y02}
    \ncline[linewidth=1pt]{x03}{y03}
    \ncline[linewidth=1pt]{x04}{y04}
    \ncline[linewidth=1pt]{x05}{y05}

    \ncline[linewidth=1pt]{s00}{t00}
    \ncline[linewidth=1pt]{s01}{t01}
    \ncline[linewidth=1pt]{s02}{t02}
    \ncline[linewidth=1pt]{s03}{t03}
    \ncline[linewidth=1pt]{s04}{t04}
    \ncline[linewidth=1pt]{s05}{t05}
    \endpspicture           \\

    Board 1 & Board 2 \\

 \end{tabular}
 \end{center}
  \caption{Bingo Closure.}
 \label{Bingen}
 \end{figure}


 Bingo closure, of course, can be considered on any $n\times n$ grid.
 A square $s$ is {\em dependent} on a set $S$ of squares
 iff $s$ completes a {\em line} --- a row, column, or diagonal --- with squares that are already in $S$.
 Let $X$ denote the $n^2$ squares on an $n \times n$ grid.
 Dependency defines a set map $\f: \Pow(X) \to \Pow(X)$ where
 $\f(S)$ is the set of all squares dependent on $S$.

 For example in Board 1 in Figure \ref{Bingen}, let $S$ be the set of 12 squares marked by solid black dots.
 Each of the 5 squares marked by open circles completes a line with
 squares in $S$. Thus $\f(S)$ is the set $A$ of squares marked by open
 circles.  Our definition of dependency has the undesirable
 peculiarity that a point {\em in} a set may fail to be dependent on
 that set. In particular it is the case in Baord 1 that no square in $S$ is dependent on $S$.  However, if
 we  look at $S \cup A$ both diagonals belong to $S \cup A$ and are
 also dependent on $S \cup A$.

 A set map is {\em isotone} provided $\f(A) \subseteq \f(B)$ whenever $A \subseteq B \subseteq X$.
 A set map is {\em expansive} provided $A \subseteq \f(A)$ for all $A \subseteq X$.
 A set map over a finite set $X$ that is isotone and expansive is {\em dolmatic\footnote{From the Turkish verb ``dolmak"
 which means ``to fill".}}. In general, the first step in computing the $\f$-closure is to produce
 the dolmatic extension $\f_*(A)$ of $\f$:
 \[
        \f_*(A) := A \cup \bigcup \{\f(S): S \subseteq A \}.
 \]
 In the case of Bingo closure the dependency map $\f$, as defined above,
 is isotone already so we can get its dolmatic extension as $\f_*(A) := A \cup \f(A)$.
 The importance of having a dolmatic function is that its iterates
 are always increasing and hence, in the finite case, eventually
 stabilize at the closure.  Here a set is {\em closed} iff it
 contains all points dependent on it, and the {\em closure} of a set
 $S$ is the smallest closed set containing $S$.  These ideas are
 discussed in a rather different setting in \cite{JN} and are
 treated in abstract generality in \cite{Haus, Jam}.
 In this paper, we answer the question for Bingo closure:

 \textsf{What is the maximum number of times the dolmatic map
        must be applied to obtain the closure of a set?}
  \rmv{
        \begin{quote}   What is the maximum number of times the dolmatic map
        must be applied to obtain the closure of a set?
        \end{quote}
        }

 \section{Maximum Depth for the Bingo Closure}

 Throughout the rest of this paper $X$ will denote set of $n^2$ squares on an $n \times n$ grid.
 The closure is the Bingo closure described above.  The Bingo
 closure of a set $S$ will be denoted by $\sC(S)$.
 The {\em depth} of a set $S$ is the number of iterations of the
 dolmatic dependency map $\f_*$ required to obtain the closure $\sC(S)$ of $S$.
 In other words, the {\em depth} of $S$ is the smallest $d$ such that $\f_*^{d+1}(S) = \f_*^d(S)$.

 A set $S$ {\em spans} $X$ provided its closure is $X$ --- that is, $\sC(S) = X$.
 The little lemma below helps to show that the maximum depth can be
 achieved only by a spanning set.

 \begin{lemma}
 \label{prop}
 If $K$ is closed but not all of $X$, then there are at least 4
 points of $X$ not in $X$ and at least 4 lines not contained in $K$.
 \end{lemma}

 \begin{proof}
 Suppose $p$ is not in $K$. Since $K$ is closed, $K$ must be missing another point on the row and a point on the column containing $p$.  Say, $r$ and $c$ are also
 missing from $K$ in the row and column of $p$.  Since $r$ is
 missing from $K$, there is another point $q$ missing from its column.
 Clearly $q \ne c$ since $p$ and $r$ are different points in the
 same row and hence lie in different columns.  Thus 4 points are
 missing.  The rows through $p$ and $c$ are not in $K$ as are the
 columns through $p$ and $r$.  These are 4 lines not contained in $K$.
 \end{proof}

 \begin{figure}
 \label{odd}
    \begin{center}
      \begin{tabular}{| c | c || c | c | c | c | c || c | c |}
        \hline
          3 & \bl & \bl & \bl & \bl & \bl & \bl & 4   & \bl \\ \hline
        \bl &   7 & \bl & \bl &  8  & \bl & \bl & \bl & \bl \\ \hline   \hline
        \bl & \bl &  10 & \bl &  9  & \bl & \bl & \bl & \bl \\ \hline
        \bl & \bl &  17 & \bl & \bl & 16  & \bl & \bl & \bl \\ \hline
        \bl & \bl &  18 & \bl & \bl & 17  & \bl & \bl & \bl \\ \hline
        \bl & \bl & \bl & 13  & \bl & \bl & 12  & \bl & \bl \\ \hline
        \bl & \bl &  15 & 14  & \bl & \bl & 11  & \bl & \bl \\ \hline   \hline
        \bl &  6  & \bl & \bl & \bl & \bl & \bl &  5  & \bl \\ \hline
         2  & \bl & \bl & \bl & \bl & \bl & \bl & \bl &  1  \\
        \hline
      \end{tabular}
    \end{center}
    \caption{A $9\times9$ set of maximum depth}
    \label{odd}
\end{figure}

 \begin{theorem}
 \label{mainthm}
 a) The depth of any non-spanning set $S$ on an $n\times n$ board is at most $2n-2$.

 b) The depth of a spanning set $S$ at most $2n$.
 \end{theorem}

 \begin{proof}
 Let $K = \sC(S)$ be the Bingo closure of $S$ and suppose the depth of $S$ is $\la$.
 Each iteration of the dependency map completes at least one line in $K$.  Obviously, once a line
 is completed, it cannot be completed again.  Thus the number of steps required to obtain the closure $K$
 is bounded by the number $\la$ of lines contained in $K$.  That is, $d \le \la$.
 The number of lines in $X$ is $2n+2$, so if $S$ is not spanning, the number of lines in $K$ is at most
 $2n-2$.  Therefore (a) is established, so we can assume $S$ is
 spanning and $K = X$.

 Now consider the last element $z$ added in forming the closure $\sC(S) = X$ of $S$. If $z$ is to be last,
 then it must complete all lines in $X$ through $z$.  There are at least
 two such lines (a row and a column) that are not complete before $z$ is added.
 Thus two lines are used in the same iteration, giving an upper bound of $2n+1$.

 Now, in this case the last element to be added to $S$ comes from two
 lines simultaneously. Hence two elements of $S$ must also be added
 to $S$ in the penultimate iteration. This also uses two lines in the same
 iteration, and so the upper bound becomes $2n$.
 \end{proof}

 Board 2 of Figure \ref{Bingen} shows a spanning set of depth 10 in
 the traditional  Bingo board with $n = 5$.  The spanning set $S$
 consists of those squares marked with black dots.  The numbers in
 the other squares indicate the order in which they are picked up by
 the dolmatic dependency map in forming the closure.

 \begin{figure}
     \begin{center}
      \begin{tabular}{| c | c || c | c | c | c | c | c || c | c |}
        \hline
          3 & \bl & \bl & \bl & \bl & \bl & \bl & \bl & 4   & \bl \\ \hline
        \bl &   7 & \bl & \bl & \bl &  8  & \bl & \bl & \bl & \bl \\ \hline  \hline
        \bl & \bl &  10 & \bl & \bl &  9  & \bl & \bl & \bl & \bl \\ \hline
        \bl & \bl &  19 & \bl & \bl & \bl &  18 & \bl & \bl & \bl \\ \hline
        \bl & \bl &  20 & \bl & \bl & \bl &  19 & \bl & \bl & \bl \\ \hline
        \bl & \bl & \bl & 14  &  15 & \bl & \bl & \bl & \bl & \bl \\ \hline
        \bl & \bl & \bl & 13  & \bl & \bl & \bl & 12  & \bl & \bl \\ \hline
        \bl & \bl &  17 & \bl &  16 & \bl & \bl & 11  & \bl & \bl \\ \hline \hline
        \bl &  6  & \bl & \bl & \bl & \bl & \bl & \bl &  5  & \bl \\ \hline
         2  & \bl & \bl & \bl & \bl & \bl & \bl & \bl & \bl &  1  \\
        \hline
      \end{tabular}
    \end{center}
    \caption{A $10\times10$ set of maximum depth}
      \label{even}
    \end{figure}

 \begin{theorem}
 For $n\geq 5$, the Bingo board of side $n$ has a spanning set  $S$ of depth at least $2n$.
 \end{theorem}

 \begin{proof}
 The proof is by construction of such a set of depth $2n$ for each $n\geq 5$.
 We proceed by induction, constructing the larger cases from smaller ones.
 We must split into the cases of even and odd $n$. The base
 case for the odd construction is given by Board 2 in Figure \ref{Bingen}.
 Larger boards can be obtained by spiraling around the base case as
 shown in Figure \ref{odd}.  The base case is in the $5 \times 5$ rectangle
 enclosed by double lines.  The numbers inside the base case have
 been increased by 8 since those squares appear 8 steps later in the
 closure process.  The 4 new rows and and 4 new columns get used
 first.

 The situation is analogous in the even case, with the base case given in Board S of Figure \ref{small}.
 Figure \ref{even} shows the construction for $n = 10$.  Again the base case is
 in the $6 \times 6$ rectangle enclosed by double lines.
 For larger grids the same spiraling technique can be used.
 \end{proof}

 For $n<5$, the maximum depth is less than $2n$. For $n=1$ and $n=2$, the maximum depth is
 clearly just one. For $n=3$, the maximum depth is 4. This is given
 by the construction below in Board 3 of Figure \ref{small}. This could be seen
 by merely enumerating all possible $3\times 3$ boards.
 Instead however, let us consider again the last piece filled in.

 The last four squares to be filled is
 significant and is the primary technique used in this paper in
 analyzing the maximum depth of small grids. The proof technique of Theorem \ref{mainthm} shows the fact that these last four squares filled in form the corners of a
 rectangle. Hence they will be called the {\em final rectangle}.

 The comments above show that the last piece filled in comes from a
 rectangle of four, and that this rectangle must have depth at most three.
 Now to have depth three there must be a diagonal line which
 is on precisely one of the squares on the rectangle of four.
 No such rectangle exists on a $3\times 3$ board, so the depth of the final rectangle is at most 2. Hence because each
 rectangle lies on both diagonals, there are at least $4$ wasted lines.
 Hence the maximum depth is indeed $2\cdot 3+2-4=4$.

 \begin{figure}
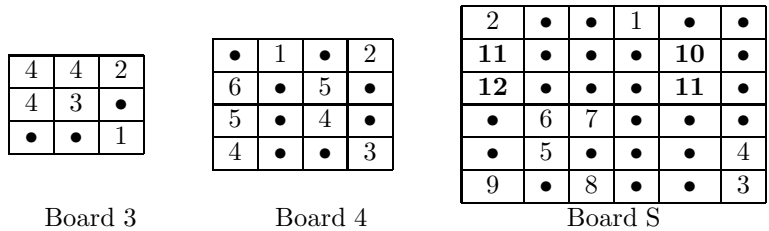

    \begin{center}
    \begin{tabular}{ccc}

      \begin{tabular}{| c | c | }
      \end{tabular}
      \begin{tabular}{| c | c | c | }
        \hline
        4 & 4 & 2 \\ \hline
        4 & 3 & \bl \\ \hline
        \bl & \bl & 1 \\
        \hline
      \end{tabular}

      \hspace{.15 in}
      &
                \begin{tabular}{| c | c | c | c |}
            \hline
            \bl &  1  & \bl & 2   \\ \hline
            6   & \bl & 5   & \bl \\ \hline
            5   & \bl & 4   & \bl \\ \hline
            4   & \bl & \bl & 3   \\
            \hline
          \end{tabular}

        \hspace{.15 in}
        &
              \begin{tabular}{| c | c | c | c | c | c |}
        \hline
        2 & \bl & \bl &  1  & \bl & \bl  \\ \hline
        {\bf 11} & \bl & \bl & \bl &  {\bf 10} & \bl  \\ \hline
        {\bf 12} & \bl & \bl & \bl &  {\bf 11} & \bl  \\ \hline
        \bl & 6  &  7 & \bl & \bl & \bl \\ \hline
        \bl & 5  & \bl & \bl & \bl & 4  \\ \hline
        9 & \bl &  8 & \bl & \bl & 3   \\
        \hline
      \end{tabular}             \\

      Board 3 & Board 4 & Board S

      \end{tabular}
    \end{center}
        \caption{Small cases}
    \label{small}
    \end{figure}

For $n=4$ the analysis is a little more tricky. This is because of the necessary overlap
of the diagonals and the final rectangle. There are several cases to worry about though.

If a rectangle is left empty four squares are left empty and there
are least $5$ (two horizontal, two vertical, and one
diagonal) wasted lines giving a maximum depth of $2\cdot 4 + 2 - 5=5$. If on the
other hand we are to generate all $16$ tiles, then the final
rectangle must overlap either both diagonals. or a single diagonal
twice. If it has has a single diagonal twice, one cannot come onto a
final rectangle via a diagonal and thus at all. This leaves only the
case that it contains both diagonals. Hence one of the squares must
be filled from two lines - one diagonal and one from horizontal or
vertical. This gives an additional two wasted lines: the diagonal, and
whatever line filled the square that used the diagonal. This gives a
maximum depth of $2\cdot 4 + 2 - 2 - 1 - 1=6$. Such a construction
is given below in Board 4 of Figure \ref{small}.


\end{document}